\documentclass[twoside,a4paper,reqno,12pt]{amsart}

\usepackage{cite}
\usepackage{mathrsfs}
 \usepackage{booktabs}
\usepackage{amsmath}
\usepackage{amsthm}
\usepackage{latexsym}
\usepackage{amssymb,bm}
\usepackage{amsfonts}
\usepackage{longtable}
\usepackage{multirow, hyperref}
\usepackage{braket}
\usepackage{array}
\usepackage{amsmath,enumerate}

\usepackage{graphics}

\def\l{\langle}
\def\r{\rangle}

\usepackage{color}

\def\NN{\mathbb N}
\def\PP{\mathbb{P}}
  
\def\Aut{{\sf Aut}} 
 
\def\PSL{{\rm PSL}}
\def\PSU{{\rm PSU}}

\def\calI{\mathcal{I}}
\def\calA{\mathcal{A}}
\def\calS{\mathcal{S}}
\def\calG{\mathcal{G}}
\def\calR{\mathcal{R}}

\def\calH{\mathcal{H}}
\def\excalR{\mathcal{R}^\mathrm{ex}}
\def\calHG{\mathcal{H}\mathcal{G}}
\def\calHR{\mathcal{H}\mathcal{R}}

\def\C{{\bf C}}
\def\Z{{\bf Z}}

\newtheorem{thm}{Theorem}[section]
\newtheorem{que}[thm]{Question}

\newtheorem{theorem}[thm]{Theorem}

\newtheorem{lemma}[thm]{Lemma}

\newtheorem{corollary}[thm]{Corollary}

\theoremstyle{remark}
\newtheorem{rem}[thm]{Remarks}

\allowdisplaybreaks[4]

\begin{document}
	\title[Proportion of chiral maps]
	{Proportion of chiral maps with automorphism group $\mathcal{S}_n$ and $\mathcal{A}_n$}

	\author{Jiyong Chen}
	\author{Yi Xiao Tang}
	
	\address{J.~Chen, School of Mathematical Sciences\\
		Xiamen University \\
		Xiamen 361005\\
		P. R. China}
	\email{chenjy1988@xmu.edu.cn}

	\address{Y.X.~Tang, School of Mathematical Sciences\\
		Xiamen University \\
		Xiamen 361005\\
		P. R. China}
	\email{cj11789@163.com}

\date\today

\begin{abstract}
    Orientably-regular maps are highly symmetric embeddings of graphs in oriented surfaces. Among them, \emph{chiral} maps are those which fail to be isomorphic to their mirror images.

    We prove that, as $n\to\infty$, chirality is generic for orientably-regular maps with automorphism groups $\calS_n$ or $\calA_n$: the proportion of chiral maps tends to $1$ in both families. We also obtain the corresponding asymptotic result for orientably-regular hypermaps with automorphism groups
    $\calS_n$ or $\calA_n$.

    A key ingredient is a sharp asymptotic generation statement: if one chooses an involution of $\calS_n$ uniformly at random and then chooses an independent uniformly random element of $\calS_n$, the probability that these two elements generate $\calS_n$ and $\calA_n$ tends to $\frac{3}{4}$ and $\frac{1}{4}$  as $n\to\infty$, respectively.
    
\end{abstract}
		\maketitle

\section{Introduction}
An \emph{oriented map} (or simply \emph{map}) $\mathcal{M}$ is an embedding of a graph $\Gamma$ into an oriented surface $\mathcal{S}$ such that each connected component of $\mathcal{S}\setminus\Gamma$ is homeomorphic to an open disk. 
An \emph{automorphism} of the map $\mathcal{M}$ is a permutation on the sets of vertices, edges and faces of $\mathcal{M}$ which is induced by an orientation-preserving homeomorphism that also preserves vertices, edges and faces. 
The \emph{automorphism group} of a map $\mathcal{M}$, denoted by $\Aut(\mathcal{M})$, is the group consisting of all automorphisms of $\mathcal{M}$. 
A map $\mathcal{M}$ is called \emph{orientably-regular} if $\Aut(\mathcal{M})$ acts transitively on the set of incident vertex-edge pairs. 

Research on orientably-regular maps has a long history, dating back to the work of Burnside and others in the late 19th century, and they continue to be a widely studied class of combinatorial structures. 
We refer the interested reader to the survey paper \cite{Siran-survey} and the references therein for further details.

A map is called \emph{reflexible} if there exists an orientation-reversing homeomorphism of the supporting surface that preserves its vertices, edges and faces.
A \emph{chiral map} is an orientably-regular map that is not reflexible.  Extensive research has been done on chiral maps; see, for example, \cite{chiral-psl3-psu3,conder-regular-jems-2010,conder-chiral-2010,Conder-chiral-2016,Jones-totally-chiral-2009}.

From the definitions of chiral and reflexible maps, one would expect that within a given class of maps (e.g., restricted by the genus of the surface or by the automorphism group), the proportion of reflexible maps should be relatively small, as reflexibility requires an additional symmetry condition. 
The sphere supports no chiral orientably-regular maps, while the torus admits both infinitely many reflexible orientably-regular maps and chiral  maps. For surfaces of hyperbolic type, Macbeath constructed infinitely many orientably-regular maps that attain the Hurwitz bound in \cite{macbeath-1961,macbeath-1969}. None of these maps are chiral. It is natural to ask what the proportion of chiral maps is within a given family of orientably-regular maps. One of the ways is to consider the family of orientably-regular maps with given automorphism groups. 
This leads to the following question.

\begin{que}\label{que-1}
    Let $\mathcal{F}$ be a family of finite groups. For a group $G\in \mathcal{F}$, let $\mathbb{P}_{ch}(G)$ denote the proportion (up to isomorphism) of chiral maps among all orientably-regular maps with automorphism group $G$.
    What is the asymptotic behavior of $\mathbb{P}_{ch}(G)$ as $|G|\to \infty$?
\end{que}

Given a finite group $G$, a pair of elements $(x,y)\in G\times G$ is called a $(m,n)$-pair of $G$ if $|x|=m$ and $|y|=n$, and is called a $(m,n)$-generating pair of $G$ if we further assume that $\l x,y\r=G$. We use $*$ to denote unspecified order. 
There is a well-known correspondence between an orientably-regular map and its automorphism group $G$ together with a $(2,*)$-generating pair of $G$. Specifically, given a group $G$ and a $(2,*)$-generating pair $(x,y)$ of $G$, there is a unique orientably-regular map associated with them, denoted by $\mathcal{M}(G,x,y)$. Conversely, any orientably-regular map $\mathcal{M}$ with automorphism group  $G$  is isomorphic to $\mathcal{M}(G,x,y)$ for some $(2,*)$-generating pair $(x,y)$. Two maps $\mathcal{M}(G,x_1,y_1)$ and $\mathcal{M}(G, x_2, y_2)$ are isomorphic if and only if there exists a group automorphism $\tau\in \Aut(G)$ such that $(x_1^\tau,y_1^\tau)=(x_2, y_2)$. Moreover, the map $\mathcal{M}(G,x,y)$ is \emph{reflexible} if and only if there is a group automorphism $\sigma\in \Aut(G)$ such that $(x^\sigma,y^\sigma)=(x, y^{-1})$. If there is no such automorphism, then this map $\mathcal{M}(G,x,y)$ and the corresponding generating pair $(x,y)$ are called \emph{chiral}. 
 See \cite{Jones-map-theory} for more details. 

Let $\Delta(G)$ be the set of $(2,*)$-generating pairs of $G$ and let $\Delta_{ch}(G)$ be the subset of chiral $(2,*)$-generating pairs.
Consider the action of $\Aut(G)$ on the set $\Delta(G)$. It is easy to see that two maps $\mathcal{M}(G,x_1,y_1)$ and $\mathcal{M}(G, x_2, y_2)$ are isomorphic if and only if they are in the same $\Aut(G)$-orbit.  Note that this action is semiregular. Hence 
\[\mathbb{P}_{ch}(G)=\frac{|\Delta_{ch}(G) |/ |\Aut(G)| }{|\Delta(G)|/|\Aut(G)|}=\frac{|\Delta_{ch}(G) | }{|\Delta(G)|}. \]
This simple observation tells us that, in order to compute $\mathbb{P}_{ch}(G)$, we only need to compute the sizes of $\Delta_{ch}(G)$ and $\Delta(G)$.

It is proved in \cite{king-2017} that all non-abelian simple groups are generated by an involution and an element of prime order. Hence, every non-abelian simple group is the automorphism group of some orientably-regular map. It is interesting to consider Question~\ref{que-1} for a given family of non-abelian simple groups, or more generally for a given family of almost simple groups. Combining \cite[Theorem 1.1]{Chiralpolyhedra}  with results from \cite{chiral-psl3-psu3}, we have that for any simple group $G$, the proportion $\mathbb{P}_{ch}(G)=0$ if and only if $G$ is isomorphic to $\calA_7$, $\PSL_2(q)$, $\PSL_3(q)$ or $\PSU_3(q)$. It was proved in \cite{Conder-chiral-2016} that for every given hyperbolic type, there exists a chiral map  with automorphism group $\mathcal{A}_n$ or $\mathcal{S}_n$ for some $n$.  See \cite{chiral-hurwitz-2024} for more constructions.

In this paper, we answer Question~\ref{que-1} for the symmetric and alternating groups. Specifically, we prove the following theorem.

\begin{theorem}\label{main-thm-1}
For the proportions $\mathbb{P}_{ch}(\mathcal{S}_n)$ and $\mathbb{P}_{ch}(\mathcal{A}_n)$ of chiral maps as defined above, we have
\[
1 - \mathbb{P}_{ch}(\mathcal{S}_n), \ 1 - \mathbb{P}_{ch}(\mathcal{A}_n) \in O\!\left(4^n \cdot n^{- 0.25n}\right),
\]
and consequently,
\[
\lim_{n \to \infty} \mathbb{P}_{ch}(\mathcal{S}_n) = \lim_{n \to \infty} \mathbb{P}_{ch}(\mathcal{A}_n) = 1.
\]
\end{theorem}

The theory of \emph{orientably-regular hypermaps} generalizes that of maps. An analogous correspondence exists between such hypermaps and generating pairs of their automorphism groups. Specifically, to any finite group $G$ and any $(*,*)$-generating pair $(x, y)$ (that is with no restriction on the orders), one can associate a unique orientably-regular hypermap $\mathcal{H}(G, x, y)$. Conversely, every orientably-regular hypermap $\mathcal{H}$ with automorphism group $G$ is isomorphic to $\mathcal{H}(G, x, y)$ for some generating pair $(x, y)$. The reflexibility condition is also similar: $\mathcal{H}(G, x, y)$ is reflexible if and only if there exists $\sigma \in \Aut(G)$ such that $(x^\sigma, y^\sigma) = (x^{-1}, y^{-1})$. 
A hypermap is called a chiral map if it is orientably-regular and not reflexible.
For details, see \cite{hypermap-chiral-2009}.

In \cite{Hypermaps}, Lucchini and Spiga pose a hypermap version of Question~\ref{que-1}. They also prove that,   for any simple group $G$, the proportion $\mathbb{P}_{ch\text{-}\mathcal{H}}(G)=0$ if and only if $G$ is isomorphic to  $\PSL_2(q)$. Here, $\mathbb{P}_{ch\text{-}\mathcal{H}}(G)$ denotes the proportion of chiral hypermaps among all orientably-regular hypermaps with automorphism group $G$. In this paper, we also answer the question of Lucchini and Spiga for symmetric and alternating groups, as presented in the following theorem.

\begin{theorem}\label{main-thm-2}
For the proportions $\mathbb{P}_{ch\text{-}\mathcal{H}}(\mathcal{S}_n)$ and $\mathbb{P}_{ch\text{-}\mathcal{H}}(\mathcal{A}_n)$ of chiral hypermaps as defined above, we have
\[
1 - \mathbb{P}_{ch\text{-}\mathcal{H}}(\mathcal{S}_n), \ 1 - \mathbb{P}_{ch\text{-}\mathcal{H}}(\mathcal{A}_n) \in O\!\left(10^n \cdot n^{0.5 - 0.5n}\right),
\]
and consequently,
\[
\lim_{n \to \infty} \mathbb{P}_{ch\text{-}\mathcal{H}}(\mathcal{S}_n) = \lim_{n \to \infty} \mathbb{P}_{ch\text{-}\mathcal{H}}(\mathcal{A}_n) = 1.
\]
\end{theorem}

Our proof of Theorem~\ref{main-thm-1} (resp. Theorem~\ref{main-thm-2}) relies  on the probability that a $(2,*)$-pair (resp. $(*,*)$-pair) generates the whole group $\calS_n$ or $\calA_n$. 
Dixon proved in \cite{dixon-1969} that the probability for a $(*,*)$-pair of $\calS_n$ to generate $\calS_n$ or $\calA_n$ tends to $1$ as $n\to \infty$.
In \cite{Liebeck-Shalev-23gen}, Liebeck and Shalev proved that a $(2,*)$-pair of $\calA_n$ generates $\calA_n$ with probability tending to $1$ as $n\to \infty$. However, there is no similar result for the probability that a $(2,*)$-pair generates $\calS_n$. We prove in the following lemma that a $(2,*)$-pair of $\calS_n$ almost surely generates either $\calS_n$ or $\calA_n$  as $n$ tends to $\infty$. Consequently, the probabilities of generating $\calS_n$ and $\calA_n$ tend to $\frac{3}{4}$ and $\frac{1}{4}$, respectively. Our proof follows the strategy used in \cite{dixon-1969}.
% , and reproof Liebeck and Shalev's result in \cite{Liebeck-Shalev-23gen} for $\calA_n$.
% this probability  tends to $\frac{3}{4}$ in the following theorem, which is a byproduct of this paper. 

\begin{theorem}\label{thm-byproduct}
    The probability that a randomly chosen involution of $\calS_n$ and a randomly chosen additional element generate either $\calS_n$ or $\calA_n$  tends to $1$  as $n\to \infty$. In particular, the probabilities that such a pair generates  $\calS_n$ and $\calA_n$ approach $\frac{3}{4}$ and $\frac{1}{4}$, respectively.
\end{theorem}

The paper is organized as follows. In Section~\ref{sec:preliminaries}, we introduce some preliminaries, mainly about enumeration methods. Section~\ref{sec:gen-probability} proves Theorem~\ref{thm-byproduct}, which is crucial for proving Theorem~\ref{main-thm-1}. In the final section, we prove our main results, namely Theorem~\ref{main-thm-1} and Theorem~\ref{main-thm-2}.

\section{Preliminaries}\label{sec:preliminaries}

In this section, we present some useful facts and lemmas, as well as  basic enumeration methods that will be used later.

\subsection{Sizes of \texorpdfstring{$\mathcal{S}_n$}{Sn} and of the set of involutions \texorpdfstring{$\mathcal{I}_n$}{In}} \ 

The following lemma is a refined version of Stirling's formula and will be used frequently in what follows.
\begin{lemma}\label{lem:stirling}
For $n\geq 1$, we have
\begin{equation}\label{stirling1.1}
\frac{\sqrt{2\pi n}n^n}{e^n}< n!<\frac{1.1\sqrt{2\pi n}n^n}{e^n}.    
\end{equation}   
\end{lemma}
\begin{proof}Robbins proved in \cite{Stirling} that for any integer $n\geq 1$, there exists $r_n\in\big(\frac{1}{12n+1},\frac{1}{12n}\big)$ such that
\[
 n!=\sqrt{2\pi n} \frac{n^n}{e^n}e^{r_n}.
\]
Hence, for $n\geq 1$,
\[
\frac{\sqrt{2\pi n}n^n}{e^n}\leq n!<\sqrt{2\pi n} \frac{n^n}{e^n}e^{\frac{1}{12}}<\frac{1.1\sqrt{2\pi n}n^n}{e^n}.
\]
This completes the proof.
\end{proof}

From now on, we always denote the set of involutions in $\mathcal{S}_n$ by $\mathcal{I}_n$ and denote the cardinality of $\mathcal{I}_n$ by $I(n)$.
% Let $\mathcal{I}_n$ be the set of involutions in $\mathcal{S}_n$ and let $I(n)$ be the cardinality of $\mathcal{I}_n$. 
The following lemma gives several basic facts about $I(n)$.

\begin{lemma}\label{basic fact}
The following statements about $I(n)$ hold:
\begin{enumerate}[\rm(1)]
\item {\rm\cite[Theorem 2.1]{Involutions}} The function $I(n)$ satisfies the recurrence
\begin{equation} \label{recur-in}    
I(n)=I(n-1)+(n-1)(I(n-2)+1),
\end{equation}
with initial conditions $I(0)=I(1)=0.$
\item {\rm\cite[Theorem~2.2]{Involutions}}  $I(n)+1=n![x^n]e^{x+0.5x^2}.$    
\item {\rm\cite[Theorem~8.1]{Involutions}}  $\displaystyle\lim_{n \to\infty}\frac{I(n)+1}{\frac{1}{\sqrt{2}}n^{\frac{n}{2}} e^{-\frac{n}{2} + \sqrt{n} -\frac{1}{4}}}=\lim_{n \to\infty}\frac{I(n)}{\frac{1}{\sqrt{2}}n^{\frac{n}{2}} e^{-\frac{n}{2} + \sqrt{n} -\frac{1}{4}}}=1.$
\end{enumerate}
\end{lemma}

Using the basic facts stated in this lemma, we obtain several further properties of $I(n)$ below.

\begin{lemma}\label{prior}
The following statements about $I(n)$ hold:
\begin{enumerate}[\rm(1)]
\item For $n\geq 2$, we have $(1+I(n))^2\geq 2n!$.
\item For $n\geq 7$, we have
\[
\frac{I(n)}{I(n-1)} < \sqrt{n+\frac{1}{4}}+\frac{1}{2}.
\]
\item For $n\geq 3$, we have
\[
\frac{I(n)}{I(n-1)} > \sqrt{n-\frac{3}{4}}+\frac{1}{2}.
\]
\item We have
\[
\lim_{n\to\infty}\frac{1}{\sqrt{n}}\frac{I(n)}{I(n-1)}=1.
\]
\item For $n\geq 5$, the function $Q(n):=\frac{I(n)}{I(n-1)}$ is strictly increasing.
\item For $n\geq 10$ and $2\leq i \leq \lfloor \frac{n-1}{2}\rfloor$, we have $I(i+1)I(n-i-1) < I(i)I(n-i)$.
\end{enumerate}
\end{lemma}

\begin{proof}
We first prove statement~(1). Direct computation shows that it holds for $n=2,3,4$. For $n\geq 5$, assume that statement~(1) holds for $n-1$ and $n-2$. Then
\begin{align*}
1+I(n)&=1+I(n-1)+(n-1)(I(n-2)+1)\\
&\geq \sqrt{2}\sqrt{(n-1)!}+(n-1)\sqrt{2}\sqrt{(n-2)!}\\
&=\sqrt{2}\sqrt{(n-1)!}(1+\sqrt{n-1})\\
&>\sqrt{2}\sqrt{(n-1)!}\sqrt{n}=\sqrt{2n!}.
\end{align*}
This proves statement~(1).

We now prove statements~(2) and~(3) simultaneously by induction on $n$. Direct computation using Magma \cite{MAGMA} shows that statement~(2) holds for $7\leq n\leq 21$ and statement~(3) holds for $3\leq n\leq 21$. For $n>21$, assume that both statements hold for $n-1$ and $n-2$. 
Note that when $n>21$, statement~(1) implies\begin{equation}
\begin{aligned}
I(n-1)
&> \sqrt{2(n-1)!}-1 \\
&> \sqrt{(n-1)(2n-4)(3n-9)\cdots(8n-64)}-1 \\
&> (n-1)^4-1>(n-1)^3.
\end{aligned}
\end{equation}
% and hence $I(n-1)>(n-1)^3$. 
By Equation~(\ref{recur-in}), we have
\begin{align}\label{mm}
\frac{I(n)}{I(n-1)}&=1+\frac{n-1}{\frac{I(n-1)}{I(n-2)}}+\frac{n-1}{I(n-1)}<1+\frac{n-1}{\frac{I(n-1)}{I(n-2)}}+\frac{1}{(n-1)^2}\notag\\
&<1+\frac{n-1}{\sqrt{ n-\frac{7}{4}} +\frac{1}{2}}+\frac{1}{(n-1)^2}.
\end{align}
We can verify that for $n\geq 20$,
\begin{equation}\label{middle}
\frac{1}{2}+\frac{n-1}{\sqrt{ n-\frac{7}{4}} +\frac{1}{2}}+\frac{1}{(n-1)^2}<\sqrt{n+\frac{1}{4}}.
\end{equation}
Combining (\ref{mm}) and (\ref{middle}), we obtain
$\frac{I(n)}{I(n-1)}<\frac{1}{2}+\sqrt{n+\frac{1}{4}}.$

On the other hand,
\begin{align*}
\frac{I(n)}{I(n-1)}&=1+\frac{n-1}{\frac{I(n-1)}{I(n-2)}}+\frac{n-1}{I(n-1)}\\
&>1+\frac{n-1}{\frac{I(n-1)}{I(n-2)}}>1+\frac{n-1}{\sqrt{n-\frac{3}{4}}+\frac{1}{2}}\\
&=\sqrt{n-\frac{3}{4}}+\frac{1}{2}.
\end{align*}
This proves statements~(2) and~(3).

Statement~(4) is an immediate corollary of statements~(2) and~(3).
We now prove statement~(5). Note that $Q(7)>Q(6)>Q(5)$. For $n\geq 7$, statements~(2) and~(3) imply that
$$Q(n+1)>\sqrt{n+\frac{1}{4}}+\frac{1}{2}>Q(n).$$
This completes the proof of statement~(5).

Finally, we prove statement~(6). A direct calculation shows that $Q(3)=Q(4)=Q(6)=3$. By statement~(5), for $n\geq 10$ and $2\leq i \leq \lfloor \frac{n-1}{2}\rfloor$ we have $Q(i+1)<Q(n-i)$. Hence
$\frac{I(i+1)}{I(i)} < \frac{I(n-i)}{I(n-i-1)}$, which implies that $I(i+1)I(n-i-1) < I(i)I(n-i)$. This completes the proof of statement~(6).
\end{proof}

\subsection{Generating functions and coefficient estimates} \ 

In this subsection, we recall some preliminaries on combinatorial generating functions (see \cite{Analytic-Com2009} for more details). The ordinary generating function (OGF) of a sequence $\{a_n\}_{n\in\NN}$ is the formal power series $A(z)=\sum_{n=0}^{\infty} a_n z^n$. We write $[z^n]A(z)$ for the coefficient of $z^n$ in $A(z)$. The OGF method will be used frequently in this paper. The following Lemma~\ref{19} is a standard exercise in analytic combinatorics and will be used to estimate coefficients of power series.

\begin{lemma}\label{19}
Let $f(z)$ be a holomorphic complex function in the disk $|z|<R$, where $R>0$. For any $0<r<R$, set $M(f,r)=\sup_{|z|=r}|f(z)|$. Then
\[\big|[z^n]f(z)\big|\le \frac{M(f,r)}{r^n}.\]
Moreover, if $[z^n]f(z)\ge 0$ for all $n\in\NN$, then
\[ [z^n]f(z)\le \frac{f(r)}{r^n}.\]
\end{lemma}

\begin{proof}
By Cauchy's integral formula,
\[
|[z^n]f(z)|=\left|\frac{f^{(n)}(0)}{n!}\right|=\left|\frac{1}{2\pi i}\int_{|z|=r}\frac{f(z)}{z^{n+1}}\,dz\right|\le \frac{1}{2\pi}\cdot\frac{2\pi r\,M(f,r)}{r^{n+1}}=\frac{M(f,r)}{r^n}.
\]
If $[z^n]f(z)\ge 0$ for all $n$, then for any $z$ with $|z|=r$ we have $$M(f,r)=\sup_{|z|=r}|f(z)|\le \sum_{n\ge 0}[z^n]f(z)\,r^n=f(r),$$ and hence $[z^n]f(z)\le f(r)/r^n$.
\end{proof}

    Applying Lemma~\ref{19}, we obtain upper bounds for the coefficient of $x^n$ in the following two generating functions, as shown in the next lemma.

\begin{lemma}\label{216}
For any positive integer $n$, we have 
\[
[x^n]e^{x+0.5\cdot x^2}\le \frac{3^n}{n^{0.5n}}\  \text{ and }\  [x^n]e^{x+0.25\cdot x^2}\le \frac{2^n}{n^{0.5n}}.
\]
\end{lemma}

\begin{proof} 
For $n\ge 8$, applying Lemma~\ref{19} with $f(z)=e^{z+bz^2}$ and $r=\sqrt{\frac{n}{2b}}$, we obtain
\begin{align*}
[x^n]e^{x+bx^2}
&\le \frac{e^{\sqrt{\frac{n}{2b}}+0.5n}}{\left(\frac{n}{2b}\right)^{0.5n}}.
\end{align*}
Since $n\ge 8$, we have $\sqrt{\frac{n}{2b}}\le \frac{n}{4\sqrt{b}}$, and hence
\begin{align*}
[x^n]e^{x+bx^2}
&\le \frac{e^{\left(\frac{1}{4\sqrt{b}}+0.5\right)n}}{\left(\frac{n}{2b}\right)^{0.5n}}
=\frac{e^{\left(\frac{1}{4\sqrt{b}}+0.5\right)n}(2b)^{0.5n}}{n^{0.5n}}
=\frac{\left(e^{\left(\frac{1}{4\sqrt{b}}+0.5\right)}\sqrt{2b}\right)^n}{n^{0.5n}}.
\end{align*}
Let $g(b)=e^{\left(\frac{1}{4\sqrt{b}}+0.5\right)}\sqrt{2b}$. A direct calculation shows that $g\!\left(\frac{1}{2}\right)<3$ and $g\!\left(\frac{1}{4}\right)<2$. Therefore, the desired inequalities hold for $n\ge 8$. One also checks directly that they hold for $1\le n\le 7$, completing the proof.
\end{proof}

We end this section by presenting a relation between the coefficients of two special generating functions.

\begin{lemma}\label{sumtransitive}
Let $A(x)=1+\sum_{i=1}^{\infty} a_n x^n $ and $B(x)=1+\sum_{i=1}^{\infty} b_n x^n $ be two formal power series with real coefficients. Suppose that 
% for
% Let $\{a_n\}_{n\in\NN^{+}}$ be an infinite sequence of real numbers. For each positive integer $n$, define
\[
    b_n=\sum_{\sum_{i=1}^{\infty} i k_i=n}\frac{a_1^{k_1}a_2^{k_2}\cdots a_m^{k_m}\cdots}{k_1!k_2!\cdots k_m!\cdots}.
\]
holds for $n\geq 1$. 
Then, for $n\ge 2$, we have
\[
    b_n=a_n+\frac{1}{n}\sum_{k=1}^{n-1}ka_k b_{n-k}.
\]
\end{lemma}

\begin{proof}
Consider the generating function
\begin{align}\label{Un}
B(n)=1+\sum_{n=1}^{\infty}b_n x^n
&=1+\sum_{n=1}^{\infty}\left(\sum_{\sum_{i=1}^{\infty} i k_i=n}\frac{a_1^{k_1}a_2^{k_2}\cdots a_m^{k_m}\cdots}{k_1!k_2!\cdots k_m!\cdots}\right)x^n\notag\\
&=\sum_{k_1,k_2,\ldots\in\NN}\frac{(a_1x)^{k_1}(a_2x^2)^{k_2}\cdots (a_mx^m)^{k_m}\cdots}{k_1!k_2!\cdots k_m!\cdots}\notag\\
&=\prod_{m\ge 1}\left(\sum_{k_m=0}^{\infty}\frac{(a_mx^m)^{k_m}}{k_m!}\right)
=e^{a_1x+a_2x^2+\cdots+a_mx^m+\cdots}.
\end{align}
Comparing coefficients of $x^n$ gives
\[
    b_n=[x^n]e^{a_1x+a_2x^2+\cdots+a_mx^m+\cdots}.
\]
Differentiating both sides of Equation~(\ref{Un}) yields
\[
\sum_{n=1}^{\infty}n b_n x^{n-1}=\left(1+\sum_{n=1}^{\infty}b_n x^n\right)\left(a_1+2a_2x+\cdots+ma_mx^{m-1}+\cdots\right).
\]
Comparing coefficients of $x^{n-1}$, we obtain
\[
    nb_n=na_n+\sum_{k=1}^{n-1}ka_k b_{n-k}.
\]
Dividing by $n$ completes the proof.
\end{proof}

\section{The \texorpdfstring{$(2,*)$}{(2,*)}-generating probability for \texorpdfstring{$\mathcal{S}_n$}{Sn}} \label{sec:gen-probability}

We aim to prove Theorem~\ref{thm-byproduct} in this section. Equivalently, we need to compute the proportion of $(2,*)$-generating pairs among all $(2,*)$-pairs in $\calS_n$. As mentioned in the introduction, we follow the strategy of Dixon in \cite{dixon-1969}. More precisely, we first compute the proportion of $(2,*)$-pairs that generate an intransitive subgroup of $\calS_n$ in subsection~\ref{subsec:intran}. Then we compute the proportion of $(2,*)$-pairs that generate an imprimitive subgroup in subsection~\ref{subsec:imprim}. Finally, in subsection~\ref{subsec:byproduct}, using several classical results on primitive subgroups of $\calS_n$, we show that the probability that a random $(2,*)$-pair generates either $\calS_n$ or $\calA_n$ tends to $1$ as $n\rightarrow \infty$. This result immediately implies Theorem~\ref{thm-byproduct}.

\subsection{Proportion for intransitive pairs}\label{subsec:intran} \

Recall that $\calI_n$ denotes the set of involutions in $\calS_n$ and $I(n)$ denotes the cardinality of $\calI_n$. 
For $n\geq 3$, let $\PP_{\text{\rm int}}(n)$ be the proportion of $(2,*)$-pairs in $\calI_n\times \calS_n$ that generate an intransitive subgroup. That is, 
\[\PP_{\text{\rm int}}(n)=\frac{\Big|\{(x,y)\in \calI_n\times \calS_n \mid \l x,y\r \leq \calS_n \text{ is intransitive} \} \Big |}{|\calI_n\times \calS_n|}.  \]
The following lemma gives a recursive formula for $\PP_{\text{\rm int}}$; here $\PP_{\text{\rm int}}(1) $ and $ \PP_{\text{\rm int}}(2)$ are defined to be $0$ to make the formula concise.

\begin{lemma}\label{recursive-formula-rn}
The following equation holds. 
    \begin{align*}
    \PP_{\text{\rm int}}(n) &=\frac{ \displaystyle \sum_{i=1}^{n-1} (i+1)I(i)+\frac{n}{2}\displaystyle\sum_{i=2}^{n-2}I(i)I(n-i)-\displaystyle\sum_{i=1}^{n-1}iI(i)(1+I(n-i))\PP_{\text{\rm int}}(i)}{nI(n)}
\end{align*}
\end{lemma}

\begin{proof}
We denote the set $\{1,2,\dots, n\}$ by $[n]$.
Let $\mathcal{P}=\{\Omega_1, \Omega_2,\dots \}$ be a partition of $[n]$, and suppose that, for each $1\leq i \leq n$, there are $k_i$ parts in $\mathcal{P}$ of size $i$. 
We first enumerate the number of $(2,*)$-pairs of $\calS_n$ that generate a group whose orbits form the partition $\mathcal{P}$. 
Let $(x,y)$ be a $(2,*)$-pair of $\calS_n$ which generates an intransitive subgroup with set of  orbits $\mathcal{P}$.
For any orbit $\Omega\in \mathcal{P}$ of  $\l x,y\r$ of length $i$, the restriction pair   $(x|_{\Omega},y|_{\Omega})$ on $\Omega$ can be viewed as a transitive pair in $(\mathcal{I}_{i}\cup \{ 1\} )\times \mathcal{S}_{i}$. Then  $x|_{\Omega}$ is either an involution or the identity and when the latter holds, $y|_{\Omega}$ must be an $i$-cycle. Hence, the number of pairs $(x|_{\Omega},y|_{\Omega})$ generating  a transitive subgroup is  
$$(1-\PP_{\text{\rm int}}(i)) I(i) i!+(i-1)!.$$
This yields that the number of pairs $(x,y)\in (\mathcal{I}_n\cup\left\{1\right\})\times \mathcal{S}_n$ such that the orbits of $\l x,y\r$ form the partition $\mathcal{P}$ is 
$$\prod_{i=1}^{n}\Big((1-\PP_{\text{\rm int}}(i)) I(i) i!+(i-1)!\Big)^{k_i}=\prod_{i=1}^{n}(i-1)!^{k_{i}}\prod_{i=1}^{n}\Big((1-\PP_{\text{\rm int}}(i))iI(i)+1\Big)^{k_i}.$$

Now, note that the above enumeration depends only on the parameters $k_1,k_2, \dots , k_n$ and the number of partitions of $[n]$ with the same parameters is  $n!/\Big(\prod_{i=1}^n k_i!\prod_{i=1}^n i!^{k_i}\Big ).$
This gives
\begin{align*}
  (I(n)+1)n!= & |(\calI_n\cup \{1\})|\cdot |\calS_n| \\
   =&\sum_{\sum_{i=1}^{n} ik_i=n}\left(\prod_{i=1}^{n}(i-1)!^{k_{i}}\prod_{i=1}^{n}((1-\PP_{\text{\rm int}}(i))iI(i)+1)^{k_i}\frac{n!}{\prod_{i=1}^n k_i!\prod_{i=1}^n i!^{k_i}} \right)
   \\=&\sum_{\sum_{i=1}^{n} ik_i=n} \Bigg(\prod_{i=1}^{n}\Big ((1-\PP_{\text{\rm int}}(i))iI(i)+1\Big)^{k_i}\frac{n!}{\prod_{i=1}^n k_i!\prod_{i=1}^n i^{k_i}} \Bigg)
   \\=&n!\sum_{\sum_{i=1}^{n} ik_i=n} \Bigg(\prod_{i=1}^{n} \frac{\Big (\frac{(1-\PP_{\text{\rm int}}(i))iI(i)+1}{i}\Big)^{k_i}}{ k_i! } \Bigg).
\end{align*}
Set $a_i=\frac{(1-\PP_{\text{\rm int}}(i))iI(i)+1}{i}$ and  $b_i=I(i)+1$ for $i \in \mathbb{N}^+$. 
By Lemma~\ref{sumtransitive}, we obtain the following recursive formula for $I(n)$:
\begin{align*}
n(I(n)+1)&=na_n+\sum_{k=1}^{n-1}(I(n-k)+1)ka_k\\&=(1-\PP_{\text{\rm int}}(n))nI(n)+1+\sum_{k=1}^{n-1}(I(n-k)+1)((1-\PP_{\text{\rm int}}(k))kI(k)+1)\\
&=(1-\PP_{\text{\rm int}}(n))nI(n)+1+n-1+\sum_{k=1}^{n-1}I(n-k)+\sum_{k=1}^{n-1}(1-\PP_{\text{\rm int}}(k))kI(k) \notag \\ & \quad +\sum_{k=1}^{n-1}kI(k)I(n-k)-\sum_{k=1}^{n-1}kI(k)I(n-k)\PP_{\text{\rm int}}(k)\\
&=\sum_{k=1}^{n}kI(k)-\sum_{k=1}^{n}kI(k)\PP_{\text{\rm int}}(k)+n+\sum_{k=1}^{n-1}I(k)+\frac{n}{2}\sum_{k=2}^{n-2}I(k)I(n-k)\notag \\
&\quad -\sum_{k=1}^{n-1}kI(k)I(n-k)\PP_{\text{\rm int}}(k)\\
&=nI(n)+n-nI(n)\PP_{\text{\rm int}}(n)+\sum_{k=1}^{n-1}(k+1)I(k)+\frac{n}{2}\sum_{k=2}^{n-2}I(k)I(n-k) \\ &\quad-
\sum_{k=1}^{n-1}kI(k)\PP_{\text{\rm int}}(k)(1+I(n-k)).
\end{align*}
Adding $nI(n)\PP_{\text{\rm int}}(n)-nI(n)-n$ to both sides of the above equation, we get
$$nI(n)\PP_{\text{\rm int}}(n)=\sum_{k=1}^{n-1}(k+1)I(k)+\frac{n}{2}\sum_{k=2}^{n-2}I(k)I(n-k)-
\sum_{k=1}^{n-1}kI(k)\PP_{\text{\rm int}}(k)(1+I(n-k)).$$
Dividing both sides by $nI(n)$ leads to the result.
\end{proof}

By this recursive formula, we give upper  lower bounds for $\mathbb{P}_{\text{\rm int}}(n)$.

\begin{theorem}\label{14}
For any positive constant $\varepsilon$,
there is an integer $N$ such that for any $n>N$, $$\frac{1-\varepsilon}{n^{1.5}}<\PP_{\text{\rm int}}(n)<\frac{2.5+\varepsilon}{n^{0.5}}. $$ 
\end{theorem}

\begin{proof}
By Lemma~\ref{prior}(6), for $n\geq 10$, the following inequality holds.
\begin{equation}\label{333}
I (2)I(n-2)>I(3)I(n-3)>\ldots> I(\lfloor\frac{n+1}{2}\rfloor) I(\lceil\frac{n-1}{2}\rceil) .  
\end{equation}
By Equation~(\ref{recur-in}), 
\begin{align*}
    I(n+1)-I(2)& =\sum_{i=2}^{n} \left( I(i+1)-I(i) \right) =\sum_{i=2}^{n} i \left( I(i-1)+1 \right)
    \\ &= \frac{(n+1)n}{2}-1+ \sum_{i=1}^{n-1}( i +1) I(i). 
\end{align*}
This gives 
\begin{equation} \label{eqn:(i+1)I(i)}
    \sum_{i=1}^{n-1}(i+1)I(i)=I(n+1)-\frac{(n+1)n}{2}.
\end{equation}
% Note that $0<\PP_{\text{\rm int}}(n)<1$ for any $n\geq 3$. 
Now, by Lemma~\ref{recursive-formula-rn},  
\begin{align*}
    \PP_{\text{\rm int}}(n) 
    & \leq \frac{\displaystyle\sum_{i=1}^{n-1} (i+1)I (i)+\frac{n}{2}\sum_{i=2}^{n-2}I(i)I(n-i)}{nI(n)} \\
    & <\frac{\displaystyle I(n+1)-\frac{(n+1)n}{2} +\frac{n}{2}\Big(2I(2)I(n-2) +(n-5)I(3)I(n-3)\Big)}{nI(n)} \\
    & <\frac{\displaystyle I(n+1) +\frac{n}{2}\Big(2I(n-2) +(3n-15)I(n-3)\Big)}{nI(n)} \\
    & =\frac{Q(n+1)}{n}+
\frac{1}{Q(n-1)Q(n)}+\frac{1.5n-7.5}{Q(n-2)Q(n-1)Q(n)}, 
\end{align*}
where $Q(n)=\frac{I(n)}{I(n-1)}$ as defined in Lemma~\ref{prior},  the first line follows from the fact that $\mathbb{P}_{\text{\rm int}}(n) \geq 0$, and the second line follows from Inequality~\ref{333} and Equation~\ref{eqn:(i+1)I(i)}.
 By Lemma~\ref{prior}(4), 
\begin{align*}
    \varlimsup_{n\to\infty} n^{0.5}\PP_{\text{\rm int}}(n) &\leq \lim _{n\to\infty} n^{0.5} \left(\frac{Q(n+1)}{n}+
\frac{1}{Q(n-1)Q(n)}+\frac{1.5n-7.5}{Q(n-2)Q(n-1)Q(n)} \right) \\&=1+1.5=2.5.
\end{align*}
This means that for any positive constant $\varepsilon$, there is an integer $N_1$ such that for any $n>N_1$, $\PP_{\text{\rm int}}(n) \leq (2.5+\varepsilon)n^{-0.5}.$

On the other hand, since $\PP_{\text{\rm int}}(n)\leq 1$, we have
\begin{align*}
    \PP_{\text{\rm int}}(n)=&\frac{\displaystyle\sum_{i=1}^{n-1}(i+1)I(i)+\frac{n}{2}\sum_{i=2}^{n-2}I(i)I(n-i)}{nI(n)}-\sum_{i=1}^{n-1}\frac{iI(i)(1+I(n-i))\PP_{\text{\rm int}}(i)}{nI(n)}\\
    \geq &\frac{\displaystyle\sum_{i=1}^{n-1}(i+1)I(i)+\frac{n}{2}\sum_{i=2}^{n-2}I(i)I(n-i)-\sum_{i=1}^{n-1}iI(i)(1+I(n-i))}{nI(n)}\\
    =&\frac{\displaystyle\sum_{i=2}^{n-1}I(i)}{nI(n)}>\frac{I(n-1)}{nI(n)}.
\end{align*}
By Lemma~\ref{prior}(4), $$\lim\limits_{n\to +\infty}\frac{I(n-1)}{nI(n)}=\lim\limits_{n\to +\infty}\frac{1}{nQ(n)}=\frac{1}{n^{1.5}}.$$
It follows that for any positive constant $\varepsilon$, there is an integer $N_2$ such that for any $n>N_2$, $\PP_{\text{\rm int}}(n) \geq (1-\varepsilon)n^{0.5}.$ This completes the proof.
\end{proof}

\subsection{Proportion of imprimitive pairs}\label{subsec:imprim} \ 

In this subsection, we aim to show that the proportion of transitive but imprimitive pairs is rare.
Let $\PP_{\text{\rm imp}}(n)$ denote the proportion of pairs
$(x,y)\in \mathcal{I}_n\times \mathcal{S}_n$ such that $x$ and $y$ generate an imprimitive group.
That is 
\[\PP_{\text{\rm imp}}(n)=\frac{\Big|\{(x,y)\in \calI_n\times \calS_n| \l x,y\r  \text{ is transitive and imprimitive} \} \Big|}{|\calI_n\times \calS_n|}.  \]

\begin{lemma}\label{defimp}
The following inequality holds:
$$I(n)\PP_{\text{\rm imp}}(n)\leq 2\cdot n^{0.5}\sum_{1< m < n,\ m\mid n}\left(m^{0.5m}\cdot (1+I(\frac{n}{m}))^{m}  0.8^m\right).$$
\end{lemma}

\begin{proof}
 
 Recall that a partition $\mathcal{B}$ of $[n]$ is called a non-trivial block system if all parts in $\mathcal{B}$ have the same size, which is greater than 1 and less than $n$.
  If $\langle x,y\rangle$ is transitive and imprimitive, then there exists 
 a nontrivial block system $\mathcal{B}$ of $[n]$ that is preserved by $x,y$. In particular, we have $1<m<n$ and $m\mid n$. Let 
 \[\Theta_n=\left\{ (x,y)\in \calI_n\times\calS_n\mid \l x,y\r \text{ preserve some non-trivial block system of $[n]$} \right\}.\]
It follows immediately that $|\calI_n|\cdot |\calS_n|\cdot \PP_{\text{\rm imp}}(n) \leq |\Theta_n|$ and 
\[ |\Theta_n| \leq \sum_{1< m < n,\ m\mid n}\left( \sum_{|\mathcal{B}|=m} |\{(x,y)\in \calI_n\times\calS_n\mid  x,y \text{ preserve $\mathcal{B}$}\}| \right).\]

In the following, to simplify the expressions, we always use $d$ to denote $n/m$.
Now, suppose $\mathcal{B}=\left\{\Omega_1,\Omega_2\ldots \Omega_m\right\}$ is a non-trivial block system. 
We need to enumerate the number of $(2,*)$-pairs that preserve $\mathcal{B}$. 
Note that all permutations  preserving $\mathcal{B}$ form a subgroup $H\cong \calS_d\wr \calS_m$ of $\calS_n$. 
Let $x$ be an involution in $H$. Then $x$ induces a permutation $x_\mathcal{B}$ on $\mathcal{B}$. Suppose that  $x_\mathcal{B}$ is a product of $p$ disjoint transpositions. 
That is $x$ swaps $p$ pairs of blocks and fixes the remaining $m-2p$ blocks. 
If $x$ swaps two blocks $\Omega_i$ and $\Omega_j$, its action on  $\Omega_i\cup \Omega_j$ is uniquely determined by a bijection between $\Omega_i$ and $\Omega_j$. 
If $x$ fixes a block $\Omega_k$, then $x|_{\Omega_k}$ is either an identity or an involution. There are $\frac{m!}{(m-2p)!2^pp!}$ ways to choose which  $p$ pairs of blocks are swapped.  This gives
\begin{align}
   & |\{(x,y)\in \calI_n\times\calS_n\mid  x,y \text{ preserve $\mathcal{B}$}\}|
   =|H |\cdot|\calI_n \cap H |   \notag
   \\=&(d!)^m\cdot m!\cdot \left(\left( \sum_{p=0}^{\lfloor 0.5m\rfloor}  \frac{m!}{(m-2p)!2^pp!}\cdot (d!)^p \cdot (1+I(d))^{m-2p} \right)-1\right) \notag
   \\<&(d!)^m\cdot m!\cdot \left( \sum_{p=0}^{\lfloor 0.5m\rfloor} \frac{m!}{(m-2p)!2^pp!}\cdot (d!)^p \cdot (1+I(d))^{m-2p} \right) \notag
   \\=& (d!)^m\cdot (m!)^2\cdot (1+I(d))^{m}  \left( \sum_{p=0}^{\lfloor 0.5m\rfloor} \frac{1}{(m-2p)!p!}\cdot \left(\frac{d!}{2(1+I(d))^{2}}\right) ^p \right) \notag
   \\<& (d!)^m\cdot (m!)^2\cdot (1+I(d))^{m}  \left( \sum_{p=0}^{\lfloor 0.5m\rfloor} \frac{0.25^p}{(m-2p)!p!} \right) \label{eqn-1}
   \\=& (d!)^m\cdot (m!)^2\cdot (1+I(d))^{m}  [x^m]e^{x+0.25x^2}  \notag
   \\<& (d!)^m\cdot (m!)^2\cdot (1+I(d))^{m}  \frac{2^m}{m^{0.5m}}, \label{eqn-2}
\end{align}
where Inequality~(\ref{eqn-1}) and Inequality~(\ref{eqn-2}) follow from Lemma~\ref{prior}(1) and Lemma~\ref{216}, respectively. Moreover, we have 
\begin{align*}
    n!\cdot I(n)\cdot\mathbb{P}_{\text{\rm imp}}(n)
    =& |\calI_n|\cdot |\calS_n|\cdot \PP_{\text{\rm imp}}(n) \leq  |\Theta_n|
    \\ < &\sum_{1< m < n,\ m\mid n}\left( \sum_{|\mathcal{B}|=m} (d!)^m\cdot (m!)^2\cdot (1+I(d))^{m}  \frac{2^m}{m^{0.5m}}\right)
    \\ < &\sum_{1< m < n,\ m\mid n}\left( n!\cdot m!\cdot (1+I(d))^{m}  \frac{2^m}{m^{0.5m}}\right).
\end{align*}
% The final result follows upon dividing both sides of the above inequality by $n!$. 
By dividing both sides of the above inequality by $n!$ and applying Lemma~\ref{lem:stirling}, we have 
\begin{align*}
     I(n)\cdot\mathbb{P}_{\text{\rm imp}}(n)
     < &\sum_{1< m < n,\ m\mid n}\left(1.1\sqrt{2\pi m}\cdot \frac{m^m}{e^m}\cdot (1+I(d))^{m}  \frac{2^m}{m^{0.5m}}\right)
     \\< &1.1\sqrt{\pi }\cdot n^{0.5}\sum_{1< m < n,\ m\mid n}\left(m^{0.5m}\cdot (1+I(d))^{m}  \frac{2^m}{e^{m}}\right)
     \\< &2\cdot n^{0.5}\sum_{1< m < n,\ m\mid n}\left(m^{0.5m}\cdot (1+I(d))^{m}  0.8^m\right).
\end{align*}
This completes the proof.
\end{proof}

Now, by applying Lemma~\ref{defimp}, we can obtain an upper bound for $\mathbb{P}_{\text{\rm imp}}(n)$ for sufficiently large $n$.

\begin{theorem}\label{thm:imp}
    There is an integer $N$ such that for $n>N$, we have
    \[\mathbb{P}_{\text{\rm imp}}(n)< 3n \cdot 0.83^n \cdot  e^{-\sqrt{n}} .\]
\end{theorem}

\begin{proof}
  As in the proof of the previous lemma, let $d=n/m$.
    By Lemma~\ref{basic fact}(3), there exists an integer $N_1$, such that for any $d\geq N_1$, 
    \begin{equation*}
        0.99<\frac{I(d)}{\frac{1}{\sqrt{2}} d^{0.5d}e^{-0.5d+\sqrt{d}-0.25}}  \hspace{3ex}\text{and}\hspace{3ex}   \frac{1+I(d)}{\frac{1}{\sqrt{2}} d^{0.5d}e^{-0.5d+\sqrt{d}-0.25}}<1.01.
    \end{equation*}
    Let $N_2=\max\{N_1,100\}$ and let 
    \begin{equation*}
        G(n,d) =  (\frac{n}{d})^{0.5n/d}\cdot (1+I(d))^{n/d}  0.8^{n/d}= m^{0.5m}\cdot (1+I(d))^{m}  0.8^m.
    \end{equation*}
    
    If $d\geq N_2$, then 
    \begin{align*}
        G(n,d) & < m^{0.5m}\cdot (\frac{1.01\cdot 0.8}{\sqrt{2} \cdot e^{0.25}} d^{0.5d} e^{-0.5d+d/\sqrt{100}})^m  
       \\ & < m^{0.5m}\cdot 0.45^m ( d^{0.5d} e^{-0.5d+d/\sqrt{100}})^m  
       \\ & < m^{0.5m}\cdot 0.45^m ( (\frac{n}{m})^{0.5d} 0.68^d)^m  
       \\ & < 0.68^n \cdot n^{0.5n} \cdot m^{0.5m-0.5n}\cdot 0.45^m .
    \end{align*}
For a given $n$, by taking the derivative, we can see that the function $m^{0.5m-0.5n}\cdot 0.45^m$ attains its maximum on the  the interval $[2, n/N_2]$ at $m=2.$ That is 
\begin{equation}
    \label{eqn:d>N2}
    G(n,d) <0.68^n \cdot n^{0.5n} \cdot 2^{1-0.5n}\cdot 0.45^2 = (0.45^2\cdot 2)(\frac{0.68}{\sqrt{2}})^n\cdot n^{0.5n} <0.41\cdot 0.5^n\cdot n^{0.5n}.
\end{equation}

 If $d<N_2$, then $G(n,d) < (0.5n)^{0.25n}\cdot (1+I(N_2))^{0.5n} $ since $m\leq 0.5n$. 
Note that there exists an integer $N>N_2$, such that for $n>N$, we have \[ 1+I(N_2)<n^{0.02}\hspace{3ex}\text{and}\hspace{3ex} 0.5^{0.25n}n^{0.26n}<0.41\cdot 0.5^n\cdot n^{0.5n}. \] 
It follows that 
\begin{align*}
    G(n,d)&<  (0.5n)^{0.25n}\cdot (n^{0.02})^{0.5n}=0.5^{0.25n}n^{0.26n}<0.41\cdot 0.5^n\cdot n^{0.5n}
\end{align*}
whence $n>N$ and $d<N_2.$

Now, since the number of positive divisors of $n$ is at most $2\cdot n^{0.5}$, by Lemma~\ref{defimp}, we have 
\begin{align*}
   \PP_{\text{\rm imp}}(n) &\leq  \frac{1}{I(n)}\cdot 2\cdot n^{0.5}\sum_{1< d < n,\ d\mid n}G(n,d)= \frac{1}{I(n)}\cdot 4 n \cdot \left(0.41\cdot 0.5^n\cdot n^{0.5n} \right)
   \\&< \frac{0.99}{\frac{1}{\sqrt{2}} n^{0.5n}e^{-0.5n+\sqrt{n}-0.25}} \cdot 4 n \cdot \left(0.41\cdot 0.5^n\cdot n^{0.5n} \right)
   \\&<(0.99\cdot \sqrt{2}\cdot e^{0.25}\cdot 0.41\cdot 4)\cdot  \frac{n\cdot 0.5^n \cdot n^{0.5n}\cdot e^{0.5n}}{ n^{0.5n}e^{\sqrt{n}}} 
   \\&<3n \cdot 0.83^n \cdot  e^{-\sqrt{n}}.
\end{align*}
This completes the proof.
\end{proof}

\subsection{\texorpdfstring{Proportion of generating pairs of $\mathcal{A}_n$ and $\mathcal{S}_n$}{Proportion of generating pairs of An and Sn}}\label{subsec:byproduct} \ 

To prove Theorem~\ref{thm-byproduct}, we need to consider the proportion of $(2,*)$-pairs in $\calS_n$ that generate neither $\calS_n$ nor $\calA_n$. For this purpose,  we need the following  two key lemmas proved by Dixon and Jordan.

\begin{lemma}\label{113}
{\rm\cite[Lemma~3]{dixon-1969}}
For a prime $q\leq n-3$, define $\mathcal{C}_{q,n}$ to be the set of 
permutations in $\mathcal{S}_n$ whose cycle decomposition has exactly one cycle of length $q$ and all other cycles of length relatively prime to $q$. 
Let $\mathcal{T}_n=\bigcup_{q} \mathcal{C}_{q,n}$, where the union is over all primes $q$ with $(\ln n)^2\leq q\leq n-3$. Then there exists an integer
$N$ such that for all $n>N$, the proportion of elements of $\mathcal{S}_n$ lying in $\mathcal{T}_n$
is at least $1-\frac{4}{3\ln\ln n}$.
\end{lemma}

\begin{lemma}\label{114}
{\rm\cite[Theorem 13.9]{wielandt64}}
A primitive subgroup of $\mathcal{S}_n$ is equal to either $\mathcal{A}_n$ or $
\mathcal{S}_n$ whenever it contains
at least one permutation which is a q-cycle
for some prime $q\leq n-3$.
\end{lemma}

We now  prove that almost all $(2,*)$-pairs of $\calS_n$ generate either $\calA_n$ or $\calS_n$.

\begin{theorem}\label{115}
 Let $\PP_{\text{\rm AS}}(n)$ be the proportion of $(2,*)$-pairs in $\calS_n$ generating either $\calA_n$ or $\calS_n$.  
 Then,  for sufficiently large $n$,  we have 
 $\PP_{\text{\rm AS}}(n)  \geq 1-\frac{1.5}{\ln\ln n}$.
\end{theorem}

\begin{proof}
By Theorem~\ref{14},  there exists an integer $N_1$ such that for any $n>N_1$, 
we have $$\PP_{\text{\rm int}}(n) \leq \frac{2.6}{\sqrt{n}}.$$
By Theorem~\ref{thm:imp}, there exists an integer $N_2$ such that for any $n>N_2$, 
we have $$\PP_{\text{\rm imp}}(n) \leq  3n\cdot 0.83^n \cdot e^{-\sqrt{n}} <\frac{0.4}{\sqrt{n}}.$$
By Lemma~\ref{113}, there exists an integer $N_3$ such that for any $n>N_3$,
the proportion of $(2,*)$-pairs $(x,y)$ with $y\in \mathcal{T}_n$ is at least $1-\frac{4}{3\ln\ln n}$. Note that  there exists an integer $N_4$ such that for any $n>N_4$, $1-\frac{3}{\sqrt{n}}-\frac{4}{3\ln\ln n}> 1-\frac{1.5}{\ln\ln n}$. 
Therefore, by taking $N=\max\{N_1,N_2,N_3,N_4\}$, we have 
\[ \PP_{\text{\rm AS}}(n) \geq  1-\frac{4}{3\ln\ln n} -\PP_{\text{\rm int}}(n)-\PP_{\text{\rm imp}}(n)> 1-\frac{1.5}{\ln\ln n},\]
which completes the proof.
\end{proof}

To give the proof of Theorem~\ref{thm-byproduct}, we also need to know the proportion of even involutions in $\calI_n$. 
It seems to be a well-known fact that the proportion is $0.5$. However, we have not seen an explicit statement in the literature. Therefore, we present the proof below.

\begin{lemma}\label{117}
    The proportion of even involutions in $\calI_n$ tends to $\frac{1}{2}$ as $n\to \infty$.
\end{lemma}

\begin{proof}
Note that an even involution is a permutation that is a product of an even number of disjoint transpositions. Hence 
\begin{align*}
|\calI_n\cap \calA_n|&=\sum_{t=1}^{\lfloor 0.25n\rfloor}\frac{n!}{(n-4t)!(2t)!2^{2t}}
=\frac{n!}{2}\Big(\sum_{k=1}^{\lfloor 0.5n\rfloor}\frac{0.5^k}{(n-2k)!k!}+\sum_{k=1}^{\lfloor 0.5n\rfloor}\frac{(-0.5)^k}{(n-2k)!k!}\Big)\\
&=\frac{n!}{2}([x^n]e^{x+0.5x^2}+[x^n]e^{x-0.5x^2})-1.
\end{align*}
By Lemma~\ref{basic fact}(2), we have
\begin{align}
\lim\limits_{n\to\infty} \frac{|\mathcal I_n \cap \mathcal A_n|}{|\mathcal I_n|}&=\lim\limits_{n\to\infty} \frac{n!\Big (\frac{[x^n]e^{x+0.5x^2}+[x^n]e^{x-0.5x^2}}{2}\Big)-1}{n![x^n]e^{x+0.5x^2}-1}=\lim\limits_{n\to\infty} \frac{n!\Big(\frac{[x^n]e^{x+0.5x^2}+[x^n]e^{x-0.5x^2}}{2}\Big)}{n![x^n]e^{x+0.5x^2}}\notag\\
&=0.5+0.5\lim\limits_{n\to\infty} \frac{[x^n]e^{x-0.5x^2}}{[x^n]e^{x+0.5x^2}}. \notag
\end{align}
Let $f(z)=e^{z-0.5z^2}$. Then for any $r>0$, 
\[ M(f,r):=\sup_{|z|=r}|f(z)|=\sup_{\theta\in [0,2\pi)}\Big|f(re^{i\theta})\Big|=\sup_{\theta\in [0,2\pi)}e^{-(r\cos\theta)^2+r\cos\theta +0.5r^2}
<e^{0.5r^2+0.25}.\]
 Applying Lemma~\ref{19} with $r=\sqrt{n}$, we obtain 
$\Big|[z^n]f(z)\Big|\leq \frac{e^{0.5n+0.25}}{n^{0.5n}}.$
Note that 
$$\Bigg|\frac{[x^n]e^{x-0.5x^2}}{[x^n]e^{x+0.5x^2}}\Bigg|\leq \frac{n!\frac{e^{0.5n+0.25}}{n^{0.5n}}}{[x^n]e^{x+0.5x^2} }, $$
and by Lemma~\ref{basic fact}(2) and Lemma~\ref{basic fact}(3),
$$\lim\limits_{n\to\infty}\frac{n!\frac{e^{0.5n+0.25}}{n^{0.5n}}}{[x^n]e^{x+0.5x^2}}=\lim\limits_{n\to\infty}\frac{\sqrt{2\pi n}\frac{n^n}{e^n}\frac{e^{0.5n+0.25}}{n^{0.5n}}}{\frac{n^{0.5n}}{\sqrt{2}}e^{-0.5n+\sqrt{n}-0.25}}=0. $$
Thus, $\lim\limits_{n\to\infty} \frac{[x^n]e^{x-0.5x^2}}{[x^n]e^{x+0.5x^2}}=0.$ 
Thus \begin{align*}
\lim\limits_{n\to\infty} \frac{|\mathcal I_n \cap \mathcal A_n|}{|\mathcal I_n|}=0.5+0.5\lim\limits_{n\to\infty} \frac{[x^n]e^{x-0.5x^2}}{[x^n]e^{x+0.5x^2}}=0.5,
\end{align*}
which completes the proof. 
\end{proof}

\begin{proof}[Proof of Theorem~\ref{thm-byproduct}]
    The first half of the theorem follows directly from Theorem~\ref{115}. 

    Let $\PP_{\rm A}(n)$ and $\PP_{\rm S}(n)$ be the proportions in $(2,*)$-pairs of $\calS_n$ that generate $\calA_n$ and $\calS_n$, respectively. Then 
$\PP_{\rm S}(n)\leq 1- \frac{|\calI_n\cap\calA_n||\calA_n|}{|\calI_n||\calS_n|}$ and $\PP_{\rm A}(n)\leq \frac{|\calI_n\cap\calA_n||\calA_n|}{|\calI_n||\calS_n|}$. By Lemma~\ref{117}, we have $\displaystyle\varlimsup_{n\to\infty}\PP_{\rm A}(n)\leq \frac{1}{4}$ and $\displaystyle\varlimsup_{n\to\infty}\PP_{\rm S}(n)\leq \frac{3}{4}.$ 
By Theorem~\ref{115},  we have 
$\displaystyle \lim_{n\to \infty} (\PP_{\rm A}(n) +\PP_{\rm S}(n)) =1.$ 
Therefore, 
\[ \varliminf_{n\to\infty}\PP_{\rm A}(n)=\varliminf_{n\to\infty}(\PP_{\rm A}(n)+\PP_{\rm S}(n)-\PP_{\rm S}(n) ) =\lim_{n\to\infty}(\PP_{\rm A}(n)+\PP_{\rm S}(n) ) -\varlimsup_{n\to\infty}\PP_{\rm S}(n) \geq \frac{1}{4}. \]
This gives $\displaystyle\lim_{n\to \infty} \PP_{\rm A}(n) =\frac{1}{4}.$ Similarly, we have 
 $\displaystyle\lim_{n\to \infty} \PP_{\rm S}(n) =\frac{3}{4}$, which completes the proof.
\end{proof}

Note that the proportion of $(2,*)$-pairs in $\calS_n$ that lie in $\calA_n$  also tends  to $\frac{1}{4}$. This leads to the following corollary.

\begin{corollary}\label{coro:liebeck-shalev}
    The probability that a randomly chosen involution in $\calA_n$ and a randomly chosen additional element generate $\calA_n$ tends to $1$ as $n\to \infty$.
\end{corollary}

\begin{rem}
Corollary~\ref{coro:liebeck-shalev} was proved by Liebeck and Shalev in \cite{Liebeck-Shalev-23gen}. Theorem~\ref{thm-byproduct} provides an alternative proof of this result.  
From the proof of this subsection, it is not difficult to see that, given Theorem~\ref{115} and Lemma~\ref{117}, Corollary~\ref{coro:liebeck-shalev} is just a straightforward exercise in elementary calculus. However, the converse is not true. All our attempts to derive Theorem~\ref{115} from Corollary~\ref{coro:liebeck-shalev} and Lemma~\ref{117} have been unsuccessful.
\end{rem}

\section{Proportion of chiral maps and chiral hypermaps
% \texorpdfstring{Proportion of chiral maps and hypermaps with automorphism group $\mathcal{S}_{n}$}{Proportion of chiral maps with automorphism group Sn}
}
\label{sec:proof-of-main-thm}

In this section, we prove Theorem~\ref{main-thm-1} and Theorem~\ref{main-thm-2}.

\subsection{ Chiral maps
% Estimation of \texorpdfstring{$\mathbb{P}_{ch}(\mathcal{S}_n)$}{Pch(Sn)} and \texorpdfstring{$\mathbb{P}_{ch}(\mathcal{A}_n)$}{Pch(An)}
} \

In this subsection, we aim to show Theorem~\ref{main-thm-1}. By the discussion in the Introduction,  we only need  to show the proportion of reflexible $(2,*)$-generating pairs of $\calS_n$ or $\calA_n$ tends to $0$ as $n$ tends to infinity. We first deal with the case where the automorphism group is $\calS_n$.
Let $\calG_n$ be the set of $(2,*)$-generating pairs of $\calS_n$ and let $\calR_n$ be the set of 
 $(2,*)$-generating pairs that correspond to reflexible maps. That is 
\begin{align*}
    \calG_n& =\left\{ 
(x,y)\in \mathcal{I}_n \times \mathcal{S}_n|\langle x,y\rangle =\mathcal{S}_n\right\},\\
\calR_n&=\left\{ 
(x,y)\in \calG_n|\text{there exists } g\in \Aut(\calS_n), \text{ such that }x^g=x,y^g=y^{-1}\right\}. 
\end{align*}
Note that the proportion $\PP_{ch}(n)$ of chiral maps with automorphism group $\calS_n$ is  equal to $1-\frac{|\calR_n|}{|\calG_n|}$.
Then Theorem~\ref{main-thm-1} is equivalent to 
\begin{equation}
    \frac{|\calR_n|}{|\calG_n|} \in O(4^n\cdot n^{-0.25n}). \label{eqn: ref-prop-sn}
\end{equation}
In order to prove Equation~\ref{eqn: ref-prop-sn}, we define the set of extended reflexible triples $\excalR_n$ as follows,
 \[\excalR_n:=\left\{(g,x,y)\in \calI_n\times\calI_n\times\calS_n| x^g=x,y^g=y^{-1}\right\}. \]
The following lemma gives the relation between  $|\excalR_n|$ and  $|\calR_n |$.

\begin{lemma}\label{21}
Let $n>6$ be an integer and suppose that $(x,y)\in \calR_n$. Then there exists $g\in \calI_n$ such that $x^g=x,y^g=y^{-1}$. 
Moreover, we have $$|\calR_n|\leq |\excalR_n|=n!\cdot (I(n)+1)\cdot \sum_{q=1}^{\lfloor 0.5n\rfloor} S(n,q),$$
where $S(n,q)=\frac{|\calI_n\cap \C_{\calS_n}(\sigma_q)|}{|\C_{\calS_n}(\sigma_q)|},$ $\sigma_q=(1,2)(3,4)\dots (2q-1,2q)$ for $1\leq q \leq \lfloor \frac{n}{2}\rfloor$.
% Then $g$ is unique. Furthermore, $g$ is an involution.
\end{lemma}

\begin{proof}
 Suppose $(x,y)\in \calR_n.$ There exists 
 $\phi\in \Aut(\calS_n)$ such that $x^\phi=x$ and $y^\phi=y^{-1}.$ Since $n>6$, we have $\Aut(\calS_n)\cong \calS_n.$ In other words, every automorphism of the group $\calS_n$ is given by an inner automorphism. Hence there exists $g\in \calS_n$ such that $gtg^{-1}=t^\phi$ for any $t\in \calS_n.$ In particular, we have  $gxg^{-1}=x$ and $gyg^{-1}=y^{-1}$. Thus $g^2x(g^2)^{-1}=x$ and $g^2 y(g^2)^{-1}=y$, which implies $g^2\in \C_{\calS_n}(x)\cap \C_{\calS_n}(y).$ Since $\l x,y\r=\calS_n$, we have
$g^2\in  \Z(\calS_n).$ Thus $g^2=1$. This completes the proof of the first part of the lemma.

Now we prove the second part of this lemma. Note that $|\calR_n|\leq |\excalR_n|$ holds by the definition of 
$\excalR_n$ and the first part of this lemma. Moreover,
for any $g\in \mathcal{I}_n$ with $q$ cycles, the number of 
elements $y\in \mathcal{S}_n$ such that $gyg^{-1}=y^{-1}$ is exactly the number of elements $y\in \mathcal{S}_n$ such that $gy$ has order $1$ or $2$. Therefore, the number of elements $y\in \mathcal{S}_n$ such that $gyg^{-1}=y^{-1}$ is $I(n)+1$. 
The number of elements $x\in \mathcal{I}_n$ such that $gxg^{-1}=x$
is $|\calI_n\cap \C_{\calS_n}(\sigma_q)|$. By Burnside's lemma, there are $\frac{n!}{|\C_{\calS_n}(\sigma_q)|}$ involutions in $\mathcal{S}_n$ that have $q$ cycles. So we have
\begin{align*}
|\excalR_n|&=\sum_{q=1}^{\lfloor 0.5n\rfloor}\frac{n!}{|\C_{\calS_n}(\sigma_q)|}(I(n)+1)|\calI_n\cap \C_{\calS_n}(\sigma_q)|\\
&=n!\cdot (I(n)+1)\cdot \sum_{q=1}^{\lfloor 0.5n\rfloor} S(n,q).
\end{align*}
This proves the second part of the lemma.
\end{proof}

The following two lemmas aim to give an upper bound for $S(n,q)$.

\begin{lemma}\label{rnex}
 For $1\leq q\leq \lfloor \frac{n}{2}\rfloor$, the following inequality holds:
 $$S(n,q)<[x^q]e^{x+0.25x^2}[x^{n-2q}]e^{x+0.5x^2}.$$
\end{lemma}

\begin{proof}
By definition and Burnside's Lemma, 
\begin{equation}\label{snq}
S(n,q)=\frac{|\calI_n\cap \C_{\calS_{n}}(\sigma_q)|}{|\C_{\calS_n}(\sigma_q)|}=
\frac{|\calI_n\cap \C_{\calS_{n}}(\sigma_q)|}{(n-2q)!2^q q!}.
\end{equation}
Now we calculate $|\calI_n\cap \C_{\calS_{n}}(\sigma_q)|.$
Note that $|\calI_n\cap \C_{\calS_{n}}(\sigma_q)|$ equals the number of involutions in the group $(\Z_2\wr \calS_q)\times \calS_{n-2q}.$  The number of involutions in $\calS_{n-2q}$ is
$I(n-2q)$, while the number of involutions in $\Z_2\wr \calS_q$ is
\begin{align*}&\sum_{t=0}^{\lfloor\frac{q}{2}\rfloor}\frac{q!}{(q-2t)!2^tt!}\cdot 2^t\cdot 2^{q-2t}-1\\
=&q!  2^q \sum_{t=0}^{\lfloor\frac{q}{2}\rfloor}\frac{0.25^t}{(q-2t)!t!}-1\\
=&q!  2^q [x^{q}]e^{x+0.25x^2}-1=q![x^q]e^{x^2+2x}-1.
\end{align*}
Hence we obtain that the number of
involutions in the group $(\Z_2\wr \calS_q)\times \calS_{n-2q}$
is 
\begin{align*}
|\calI_n\cap \C_{\calS_{n}}(\sigma_q)|&=q![x^q]e^{x^2+2x}\Big(I(n-2q)+1\Big)-1\\
&=q![x^q]e^{x^2+2x}(n-2q)![x^{n-2q}]e^{x+0.5x^2}-1\\
&<q!(n-2q)![x^q]e^{x^2+2x}[x^{n-2q}]e^{x+0.5x^2}.
\end{align*}
By Equation~(\ref{snq}), we obtain
\begin{align*}
S(n,q)&<\frac{q!(n-2q)![x^q]e^{x^2+2x}[x^{n-2q}]e^{x+0.5x^2}}{(n-2q)!2^q q!}\\
&=[x^q]e^{0.25x^2+x}[x^{n-2q}]e^{x+0.5x^2}.
\end{align*}
This completes the proof.
\end{proof}

\begin{lemma}
    \label{lem:upper-bound-S(n,q)}
    For $1\leq q\leq \lfloor \frac{n}{2}\rfloor$, we have 
    \[ S(n,q) < \left( \frac{2}{9}\right)^q\cdot  \frac{3^{1.25n} }{n^{0.25n}}.\]
\end{lemma}

\begin{proof}
    By Lemma~\ref{rnex}, we have $S(n,q)<[x^q]e^{x+0.25x^2}[x^{n-2q}]e^{x+0.5x^2}$. It follows from Lemma~\ref{216} that
    \begin{align*}
        S(n,q) & < \frac{2^q}{q^{0.5q}} \cdot \frac{3^{n-2q}}{(n-2q)^{0.5(n-2q)}}=\left( \frac{2}{9}\right)^q\cdot 3^n \cdot \frac{1}{q^{0.5q}(n-2q)^{0.5(n-2q)}}.
    \end{align*}
    Let $f(x) = \frac{1}{2} x \ln x$ for $x \ge 1$. From its derivative, we can see that $f(x)$  is a strictly increasing, positive, and convex function. By Jensen's inequality, we have
    \begin{align*}
        f(q)+f(n-2q)&\geq \frac{3}{2}\left( \frac{2}{3} f(q)+\frac{1}{3}f(n-2q)\right) \geq \frac{3}{2} f\left(\frac{n}{3}\right).
    \end{align*}
    This gives $q^{0.5q}(n-2q)^{0.5(n-2q)}=e^{f(q)+f(n-2q)} \geq e^{1.5 f\left(\frac{n}{3}\right)}=\left(\frac{n}{3}\right)^{ 0.25n}$, and 
    \begin{align*}
        S(n,q) &<\left( \frac{2}{9}\right)^q\cdot 3^{n} \cdot \frac{1}{\left(\frac{n}{3}\right)^{0.25n}}=\left( \frac{2}{9}\right)^q\cdot  \frac{3^{1.25n} }{n^{0.25n}}. 
    \end{align*}
    This completes the proof.
\end{proof}

Now, we are ready to give the proof of Theorem~\ref{main-thm-1}.

\begin{proof}[Proof of Theorem~\ref{main-thm-1}]
By Theorem~\ref{thm-byproduct}, there exists an integer $N_1$ such that for any $n>N_1$, 
\[ |\calG_n| > (0.75-0.001) \cdot |\calI_n|\cdot |\calS_n| \hspace{3ex} \text{and} \hspace{3ex}  (I(n)+1)< (1+0.001) I(n). \]
Then, from Lemma~\ref{21}, we have
\[\frac{|\excalR_n|}{|\calG_n|} \leq \frac{1}{0.749} \cdot \frac{n!\cdot (I(n)+1)\cdot \sum_{q=1}^{\lfloor 0.5n\rfloor} S(n,q)}{|\calI_n|\cdot |\calS_n|} <\frac{1.001}{0.749} \cdot \sum_{q=1}^{\lfloor 0.5n\rfloor} S(n,q).\]
Applying Lemma~\ref{rnex} yields that
\begin{equation}\label{eqn:rnex/gn}
    \frac{|\excalR_n|}{|\calG_n|} < \frac{1.001}{0.749} \cdot 
\sum_{q=1}^{\lfloor 0.5n\rfloor} \left( \frac{2}{9}\right)^q\cdot  \frac{3^{1.25n} }{n^{0.25n}} <\frac{1.001}{0.749} \cdot \frac{2}{7}\cdot  \frac{(3^{1.25})^{n} }{n^{0.25n}}<0.4\cdot 4^n \cdot n^{-0.25n}. 
\end{equation}
Note that $\PP_{\rm ch}(\calS_n)=1-\frac{|\calR_n|}{|\calG_n|} > 1-\frac{|\excalR_n|}{|\calG_n|}$. This gives 
\[1 - \mathbb{P}_{ch}(\mathcal{S}_n) \in O\!\left(4^n \cdot n^{- 0.25n}\right) \hspace{3ex} \text{and} \hspace{3ex}\lim_{n\to \infty} \PP_{\rm ch}(\calS_n) =1. \]

The results for $\PP_{\rm ch}(\calA_n)$ can be derived from those for $\PP_{\rm ch}(\calS_n)$. Let $\calG_n^\calA$ be the set of $(2,*)$-generating pairs of $\calA_n$ and let $\calR_n^\calA$ be the set of 
 $(2,*)$-generating pairs that correspond to reflexible maps. That is 
\begin{align*}
    \calG_n^\calA& =\left\{ 
(x,y)\in \mathcal{I}_n \times \mathcal{A}_n|\langle x,y\rangle =\mathcal{A}_n\right\},\\
\calR_n^\calA&=\left\{ 
(x,y)\in \calG_n^\calA|\text{there exists } g\in \Aut(\calA_n), \text{ such that }x^g=x,y^g=y^{-1}\right\}. 
\end{align*}
Note that, for $n\geq 6$, $\Aut(\calA_n)\cong \calS_n$. By the same argument used in Lemma~\ref{21}, we have $|\calR^\calA_n| \leq |\excalR_n| $. 
By Theorem~\ref{thm-byproduct}, there exists an integer $N>\max\{ N_1,6\}$, such that for any $n\geq N$, 
\[ \frac{|\calG_n|}{|\calI_n|\cdot |\calS_n|}<0.75+0.001 \hspace{3ex} \text{and} \hspace{3ex} \frac{|\calG_n^\calA|}{|\calI_n|\cdot |\calS_n|}>0.25-0.001.\]
It follows from Inequality~\ref{eqn:rnex/gn} that 
\begin{align*}
    1-\PP_{\rm ch}(\calA_n) &= \frac{|\calR_n^\calA|}{|\calG_n^\calA|}< \frac{|\excalR_n|}{|\calG_n^\calA|}  =\frac{|\calI_n|\cdot |\calS_n|}{|\calG_n^\calA|} \cdot \frac{|\calG_n|}{|\calI_n|\cdot |\calS_n|} \cdot \frac{|\excalR_n|}{|\calG_n|} \\
    &<\frac{0.751}{0.249}\cdot  0.4\cdot 4^n \cdot n^{-0.25n}<1.3\cdot 4^n \cdot n^{-0.25n}.
\end{align*}
This gives 
\[1 - \mathbb{P}_{ch}(\mathcal{A}_n) \in O\!\left(4^n \cdot n^{- 0.25n}\right) \hspace{3ex} \text{and} \hspace{3ex}\lim_{n\to \infty} \PP_{\rm ch}(\calA_n) =1, \]
which completes the proof.
\end{proof}

\subsection{ Chiral hypermaps
% Estimation of \texorpdfstring{$\mathbb{P}_{ch\text{-}\mathcal{H}}(\mathcal{S}_n)$}{Pch-H(Sn)} and \texorpdfstring{$\mathbb{P}_{ch\text{-}\mathcal{H}}(\mathcal{A}_n)$}{Pch-H(An)}
} \ 

As mentioned in the Introduction, there is an analogous theory of hypermaps to the theory of maps. To consider the proportion of chiral hypermaps with automorphism group $G$, we only need to compute the proportion of $(*,*)$-generating pairs of $G$ that can be reversed by a group automorphism. More specifically, we need to compute the following proportion $\PP_{ch\text{-}\calH}(G)=1-|\calHR(G)|/|\calHG(G)|$, where $\calHG(G)$ is the set of $(*,*)$-generating pairs of $G$ and 
\begin{align*}
    % \calHG(G)&=\{(x,y)\in G\times G\mid \l x,y\r =G\},\\
    \calHR(G)&=\left\{ (x,y)\in \calHG(G)\mid \text{there exists } g\in \Aut(G), \text{ such that } x^g=x^{-1},y^g=y^{-1}\right\}.
\end{align*}

\begin{proof}[Proof of Theorem~\ref{main-thm-2}]
    Let $G=\calA_n$ or $\calS_n$ for $n>6$. Then the automorphism group of $G$ are induced by conjugation action of $\calS_n$, and $G$ can be embedded into $\calS_n$ naturally.
For any pair $(x,y)\in \calHR(G)$, by the same argument as in Lemma~\ref{21},
there exists a unique involution $g=g_{x,y}\in  \calS_n$ such that $x^g=x^{-1}$ and $y^g=y^{-1}$. Moreover, $(xg)^2=(yg)^2=1$. Now, set $\overline{\calI}_n=\calI_n\cup \{1\}$, and
% and let
% \[ \mathcal{S}=\calI_n\times \overline{\calI}_n\times \overline{\calI}_n.\]
define a function 
\[
\begin{array}{rrcl}
   \Phi:  & \calHR(G) & \to & \calI_n\times \overline{\calI}_n\times \overline{\calI}_n\\
     & (x,y ) &\mapsto & (g_{x,y}, xg_{x,y}, yg_{x,y}). 
\end{array}\]
Obviously, $\Phi$ is an injection. Hence $|\calHR(G)|\leq |\calI_n\times \overline{\calI}_n\times \overline{\calI}_n| < (I(n)+1)^3  $. 
By \cite[Corollary]{dixon-1969}, there exists an integer $N_1>6$ such that for any $n>N_1$, $$\calHG(\calS_n)>(0.75-0.001)\cdot (n!)^2 \hspace{3ex}\text{and}\hspace{3ex} \calHG(\calA_n)>(1-0.001)\cdot (\frac{n!}{2})^2$$
Note that, by Lemma~\ref{lem:stirling}, Lemma~\ref{basic fact} and Lemma~\ref{216},
\begin{align*}
    \frac{(I(n)+1)^3}{(n!)^2} & =\frac{(n! [x^n]e^{x+0.5x^2})^3}{(n!)^2}=n! ([x^n]e^{x+0.5x^2})^3< \frac{1.1\sqrt{2\pi n}n^n}{e^n} \cdot (\frac{3^n}{n^{0.5n}})^3  \\
    &=(1.1\cdot \sqrt{2\pi})\cdot n^{0.5-0.5n}\cdot (\frac{3^3}{e})^n <2.8 \cdot n^{0.5-0.5n} \cdot 10^n .
\end{align*}
Thus, we have 
\begin{align*}
    1-\PP_{ch\text{-}\calH}(\calS_n) &=\frac{|\calHR(\calS_n)|}{|\calHG(\calS_n)|}<\frac{1}{0.749}\cdot 2.8 \cdot n^{0.5-0.5n} \cdot 10^n , \hspace{3ex} \text{and} \\
    1-\PP_{ch\text{-}\calH}(\calA_n) &=\frac{|\calHR(\calA_n)|}{|\calHG(\calA_n)|}<\frac{4}{0.999}\cdot 2.8 \cdot n^{0.5-0.5n} \cdot 10^n .    
\end{align*}
This gives 
\[1 - \PP_{ch\text{-}\calH}(\calS_n),  1 - \PP_{ch\text{-}\calH}(\calA_n)\in O\!\left(n^{0.5-0.5n} \cdot 10^n\right),  \]
which completes the proof.
\end{proof}

\vskip0.2in
\thanks{{\bf Acknowledgements.}
	% The authors are grateful to the anonymous referees for their valuable comments and suggestions that have helped to improve the paper.
    The authors would like to thanks Dr. Hongyi Huang and Dr. Yan Zhou Zhu for their valuable comments and helpful discussion. 
	This work was supported by the Fundamental Research Funds for the Central Universities (No. 20720240136). 
	}

\end{document}